\documentclass[reqno,12pt]{amsart}
\usepackage{amsfonts}
\usepackage{a4wide}	
\usepackage{bm}

\numberwithin{equation}{section}

\usepackage{amsmath,amssymb,amsthm,amsfonts}
\usepackage{mathrsfs}
\usepackage{bbm}
\usepackage{hyperref}
\hypersetup{hidelinks}
\usepackage{xcolor}
\usepackage{extarrows}

\allowdisplaybreaks[4]%

\newtheorem{lemma}{Lemma}[section]
\newtheorem{theorem}{Theorem}[section]

\newtheorem{proposition}[lemma]{Proposition}
\newtheorem{remark}{Remark}[section]
\newtheorem{corollary}[lemma]{Corollary}

\newcommand{\F}{{}_2F_1}
\date{}

\everymath{\displaystyle}
\begin{document}
	\title[Sharp power-mean comparisons for Gauss hypergeometric functions]{Sharp power-mean comparisons for Gauss hypergeometric functions}

\author{Qintao Deng}
\address{[Qintao Deng] Central China Normal University, School of Mathematics and Statistics \& Hubei Key Laboratory of Mathematical Sciences, 152 Luoyu Road, Wuhan 430079, P.R. China}
\email{qintaodeng@ccnu.edu.cn}

\author{Haoran Peng}
\address{[Haoran Peng] Central China Normal University, School of Mathematics and Statistics \& Hubei Key Laboratory of Mathematical Sciences, 152 Luoyu Road, Wuhan 430079, P.R. China}
\email{mathematicphr@mails.ccnu.edu.cn}

\author{Jiatong Zou}
\address{[Jiatong Zou] Central China Normal University, School of Mathematics and Statistics \& Hubei Key Laboratory of Mathematical Sciences, 152 Luoyu Road, Wuhan 430079, P.R. China}
\email{jiatongzou000818@mails.ccnu.edu.cn}

\date{\today}

\subjclass[2020]{
33C05, 
26D07, 
26E60 
}
\keywords{Gauss hypergeometric function, weighted power mean, sharp inequality, differential comparison}

	\maketitle

	\begin{abstract}
We determine the sharp weighted power-mean comparisons for the family
$\{H_a(r)
\}_{r\in (0,1)}$
with the logarithmic interpretation at $a=0$.
In the parameter ranges considered by Barnard, Richards and Tiedeman, we give a complete characterization of all orders $\lambda,\mu\in\mathbb{R}$ for which
$$A_\lambda(w;1,1-r)\leq H_a(r)\leq A_\mu(w;1,1-r)$$
holds for every $r\in(0,1)$. In particular, our results settle completely their two power-mean conjectures. The two sharp orders are determined by the second-order expansion at $r=0$ and the endpoint matching as $r\to1$. A weighted Wronskian identity and a sign analysis of its residual establish the global inequalities.
\end{abstract}

	
\tableofcontents	
  \section{Introduction}
  \subsection{The problem}
 The weighted power mean are defined by 
\begin{equation} \label{mean}
A_p(w;1,h)=
\begin{cases}
[w+(1-w)h^p]^{\frac{1}{p}},&p\neq 0;\\
h^{1-w},&p=0,
\end{cases}
\end{equation}
where $0<w<1$, $h>0$ and $p\in\mathbb{R}$. 
Let $c>b>0$. For $a\neq 0$, define
\begin{equation} \label{H}
H_a(r)=\F(-a,b;c;r)^{\frac{1}{a}},\quad 0<r<1.
\end{equation}
Euler's integral representation shows that the hypergeometric function ${}_2F_1(-a,b;c;r)$ in \eqref{H} is positive. Hence the real power in \eqref{H} is well defined even when $a<0$. At $a=0$, we use the logarithmic definition
\begin{equation} \label{H0}
H_0(r)=\exp\left(\frac{1}{B(b,c-b)}\int_0^1 t^{b-1}(1-t)^{c-b-1}\log(1-rt)\mathrm{d}t\right).
\end{equation}
For each fixed $r\in(0,1)$, this definition satisfies $H_0(r)=\lim_{a\to0}H_a(r)$. Throughout the paper, when $a=0$, the notation $H_a$ means $H_0$.

To define the endpoint-matched order, assume in addition that $c+a-b>0$, put
\begin{equation}\label{orders-1}
w=1-\frac{b}{c},\quad p_0=\frac{a+c}{1+c},
C_a=\frac{\Gamma(c+a-b)\Gamma(c)}{\Gamma(c-b)\Gamma(c+a)}.
\end{equation}
When $a\neq 0$, denote that
\begin{equation} \label{orders}
 p_E=\frac{a\log w}{\log C_a},
\end{equation}
and when $a=0$, define
\begin{equation} \label{pE0}
 p_E=\frac{\log w}{L_0},
\end{equation}
where $$L_0=\frac{1}{B(b,c-b)}\int_0^1 t^{b-1}(1-t)^{c-b-1}\log(1-t)\mathrm{d}t<0.$$
The additional condition $c+a-b>0$ is automatic for $a\geq0$ and 
ensures integrability of Euler's beta integral at $r=1$ when $a<0$.

Our goal is to determine the sharp weighted power-mean bounds for $H_a$ in the branches $c\geq\max(1-2a,2b)$ and $c\leq\min(1-2a,2b)$ when $a\leq1$, and in the symmetric regime $a>1$, $c=2b$, which are both stated in \cite[Section 3]{BarnardRichardsTiedeman2010}. More precisely:

\noindent\textbf{Conjecture I.}
Let $a\leq1$, $c>b>0$ and $c>b-a$. If
$c\geq \max(1-2a,2b),$
then
\begin{equation*}
A_{p_0}(w;1,1-r)\leq H_a(r)\leq A_{p_E}(w;1,1-r),\quad 0<r<1;
\end{equation*}
if
$c\le\min(1-2a,2b),$
then
\begin{equation*}
A_{p_E}(w;1,1-r)\leq H_a(r)\leq A_{p_0}(w;1,1-r),\quad 0<r<1.
\end{equation*}

\noindent\textbf{Conjecture II.}
Let $a>1$, $b>0$ and $c=2b$. Then
\begin{equation*}
A_{p_E}\left(\frac{1}{2};1,1-r\right)\leq \F(-a,b;2b;r)^{\frac{1}{a}},\quad 0<r<1.
\end{equation*}

\subsection{Main results}
Our first result gives the complete sharp classification in the two parameter branches with $a\leq1$, including the logarithmic case $a=0$. 

\begin{theorem} \label{thm1}
Let $a\leq 1$, $c>b>0$ and $c>b-a$. Let $p_0$ and $p_E$ be given by\eqref{orders-1},\eqref{orders},\eqref{pE0}.
\begin{enumerate}
\item If $c\geq \max(1-2a,2b),$
then $0<p_0\leq p_E\leq1$, and for arbitrary $\lambda,\mu\in\mathbb{R}$,
\begin{equation} \label{upclass}
A_\lambda(w;1,1-r)\leq H_a(r)\leq A_\mu(w;1,1-r),\quad \forall r\in(0,1)
\end{equation}
holds if and only if 
$$\lambda\leq p_0,\quad \mu\geq p_E.$$
If $a=1$, then $p_0=p_E=1$ and $H_1(r)=A_1(w;1,1-r)$
for every $r\in(0,1)$. And \
if $c=2b=1-2a$, then $p_0=p_E=\frac{1}{2}$ and
\begin{equation} \label{identity}
H_a(r)=A_{\frac{1}{2}}(w;1,1-r)\quad\forall r\in(0,1).
\end{equation}
Outside these two cases, one has $p_0<p_E<1$ and
\begin{equation} \label{upsharp}
A_{p_0}(w;1,1-r)<H_a(r)<A_{p_E}(w;1,1-r),\quad \forall r\in (0,1).
\end{equation}
\item If
$c\le\min(1-2a,2b),$
then $0<p_E\leq p_0<1$, and \eqref{upclass} holds if and only if
\begin{equation*}
\lambda\leq p_E,\quad \mu\geq p_0.
\end{equation*}
The equality case is again exactly $c=2b=1-2a$. Outside this case,
\begin{equation} \label{lowsharp}
A_{p_E}(w;1,1-r)<H_a(r)<A_{p_0}(w;1,1-r),\quad \forall r\in (0,1).
\end{equation} 
\end{enumerate}
\end{theorem}

Our second result gives the complete sharp classification in the range $a>1$, $c=2b$.

\begin{theorem} \label{thm2}
Let $a>1$, $b>0$ and $c=2b$. 
Then
\begin{equation} \label{symrange}
1<p_E<p_0<a.
\end{equation}
Moreover,
for arbitrary $\lambda,\mu\in\mathbb{R}$,
\begin{equation} \label{symclass}
A_\lambda\left(\frac{1}{2};1,1-r\right)\leq \F(-a,b;2b;r)^{\frac{1}{a}}\leq A_\mu\left(\frac{1}{2};1,1-r\right)\quad \forall r\in (0,1)
\end{equation}
 if and only if
$
\lambda\leq p_E, $ and $ \mu\geq p_0.
$
At the two sharp orders,
\begin{equation} \label{symsharp}
A_{p_E}\left(\frac{1}{2};1,1-r\right)<\F(-a,b;2b;r)^{\frac{1}{a}}<A_{p_0}\left(\frac{1}{2};1,1-r\right),\quad \forall r\in (0,1).
\end{equation}
\end{theorem} 

\subsection{Related works}
Standard references for hypergeometric functions, elliptic integrals and their connection with classical means include \cite{AndrewsAskeyRoy1999,BorweinBorwein1987,AndersonVamanamurthyVuorinen1997}. 
Early inequalities for Gauss hypergeometric functions and complete elliptic integrals were developed in \cite{Carlson1966,AndersonVamanamurthyVuorinen1990,AndersonVamanamurthyVuorinen1992,Anderson1995,PonnusamyVuorinen1997,AndersonQiuVamanamurthy1998,QiuVuorinen1999,AndersonQiuVamanamurthyVuorinen2000}. Later Tur\'an-type, monotonicity, log-convexity and ratio estimates for hypergeometric functions can be found in \cite{Baricz2007,KarpSitnik2009,KarpSitnik2010,Richards2019,BarnardRichardsSliheet2020}. 
Sharp mean-type estimates closer to the present problem appear in \cite{BarnardPearceRichards2000,AlzerQiu2004,BarnardRichards2007}. 
Related sharp mean inequalities and convexity and concavity properties for complete or generalized elliptic integrals were studied, for example, in \cite{WangChuQiuJiang2012,WangChuChu2019,WangHeChu2020,ZhaoHeChu2021}.

The two conjectures considered above arose from the search for simultaneous sharp power-mean bounds. For $a\leq1$ in the two parameter branches below, Richards \cite{Richards2005} proved the sharp $p_0$-bound; this result is restated in \cite[Theorem B]{BarnardRichardsTiedeman2010}. Building on this bound and numerical evidence, Barnard, Richards and Tiedeman \cite[Section 3]{BarnardRichardsTiedeman2010} formulated two companion conjectures. Their Conjecture I asks for the missing $p_E$-bound in the range $a\leq1$. Our formulation above combines that conjectured bound with Richards' known $p_0$-bound. Its original $1/a$-power formulation does not explicitly include $a=0$; formulas \eqref{H0} and \eqref{pE0} give the natural logarithmic extension. Their Conjecture II concerns $a>1$, $c=2b$, and asks for the sharp lower bound of order $p_E$ \cite[Conjecture II]{BarnardRichardsTiedeman2010}.

As supporting cases, Barnard, Richards and Tiedeman established the bounds in Conjecture I for $b=1$, $c=2$ and $-1<a<1$ in \cite[Proposition 1]{BarnardRichardsTiedeman2010}, with $a=0$ interpreted by continuity, and verified Conjecture II for $b=1$, $c=2$ in \cite[Proposition 2]{BarnardRichardsTiedeman2010}. Zhao \cite{Zhao2023} later proved Conjecture I on the curve
$a=b=\frac{1}{q}\in\left(0,\frac{1}{2}\right),\quad c=1.$

In our paper, Theorem \ref{thm1} proves Conjecture I in both parameter branches, including its logarithmic extension at $a=0$. Theorem \ref{thm2} proves Conjecture II and, in addition, supplies the sharp companion upper bound of order $p_0$. Thus the main results in Subsection 1.2 settle both conjectures in their full parameter ranges and strengthen them to complete classifications of all lower and upper order.

\subsection{Strategies and ideas}
The main difficulty is that the candidate sharp orders arise from different endpoint mechanisms: local behavior at $r=0$ selects $p_0$, while endpoint matching as $r\to1$ selects $p_E$. Neither calculation alone controls the comparison on the entire interval. By uncovering a weighted-Wronskian structure governed by a scalar residual, we connect these two endpoint constraints and obtain a complete classification of all admissible lower and upper power-mean orders. The proof has four main components.
\begin{itemize}
\item \textit{Endpoint orders.} The second-order expansion at $r=0$ forces $p_0$, while the beta-integral limit as $r\to1$ forces $p_E$.

\item \textit{Differential-equation reduction.} For $a\neq0$, the change $r=1-\mathrm{e}^{-2u}$ yields a positive solution $f$ of $\mathcal{L}f=0$ and a power-mean comparison function $g_p$ satisfying
\begin{equation*}
\mathcal{L}g_p=\frac{ag_p}{\cosh^2(pu+v_0)}S_p(u),
\end{equation*}
so the comparison is governed by the scalar residual $S_p$.

\item \textit{Weighted-Wronskian comparison.} With $m(u)=(\sinh u)^c\mathrm{e}^{(c-2b)u}$,
\begin{equation*}
\begin{aligned}
W_p(u)&=m(u)\bigl(f'(u)g_p(u)-f(u)g_p'(u)\bigr),\\
W_p'(u)&=-\frac{am(u)f(u)g_p(u)}{\cosh^2(pu+v_0)}S_p(u).
\end{aligned}
\end{equation*}
Outside the identity cases, $S_{p_0}$ has a fixed nonzero sign, while $S_{p_E}$ has one sign change and $f/g_{p_E}$ returns to its matched endpoint value. In the identity cases the residual vanishes identically.

\item \textit{Logarithmic and symmetric cases.} The case $a=0$ follows from the additive analogue. For $a>1$, $c=2b$, the method gives the $p_E$ lower bound and the sharp companion $p_0$ upper bound.
\end{itemize}

\subsection{Structure of the paper}

The paper is organized as follows. Section \ref{sec-2} is the core of our work, which develops the common analytic framework. Subsection 2.1 derives the endpoint calculations that determine $p_0$ and $p_E$ and proves the necessary restrictions on all admissible orders. Subsection 2.2 obtains the second-order differential equation, isolates the residual $S_p$, and records the auxiliary hyperbolic estimates used to determine its sign. Subsection 2.3 converts this sign information into global inequalities through the weighted-Wronskian comparison, while Subsection 2.4 establishes the additive logarithmic analogue for $a=0$. Sections 3 and 4 prove 
Theorems \ref{thm1} and Theorem \ref{thm2}, respectively. Appendix A contains the auxiliary proofs.

\section{Endpoint orders and the comparison principle} \label{sec-2}
Throughout this section, unless otherwise stated, we assume that
$c>b>0$. Whenever $a<0$ and the endpoint
$r\to1^-$ is considered, we additionally assume that $c+a-b>0$.

\subsection{Endpoint orders and the beta representation}
The purpose of this subsection is to identify the orders forced by the two endpoints and thereby establish the necessary parts of the sharp classifications.

On $(0,1)$ introduce the normalized beta measure
\begin{equation} \label{beta}
\mathrm{d}\nu(t)=\frac{t^{b-1}(1-t)^{c-b-1}}{B(b,c-b)} \mathrm{d}t.
\end{equation}
Thus $\nu((0,1))=1$. For later use, the first two moments are
\begin{equation} \label{moments}
\int_0^1 t \mathrm{d}\nu(t)=\frac{B(b+1,c-b)}{B(b,c-b)}=\frac{b}{c},\quad \int_0^1 t^2 \mathrm{d}\nu(t)=\frac{B(b+2,c-b)}{B(b,c-b)}=\frac{b(b+1)}{c(c+1)}.
\end{equation}

Euler's integral representation gives, for $a\neq 0$,
\begin{equation} \label{euler}
H_a(r)^a=\int_0^1(1-rt)^a \mathrm{d}\nu(t).
\end{equation}
Formula \eqref{euler} is Euler's integral representation; see \cite[Eqs. (15.2.1), (15.2.2), and (15.6.1)]{DLMF}.
For $a=0$, formula \eqref{H0} reads
\begin{equation*}
\log H_0(r)=\int_0^1\log(1-rt) \mathrm{d}\nu(t).
\end{equation*}

The next expansion identifies the order forced at the origin and will be used repeatedly for local sharpness.

\begin{lemma}
Let $c>b>0$, and fix $a,p\in\mathbb{R}$. With $H_0$ understood as in \eqref{H0}, one has, as $r\to0^+$,
\begin{equation} \label{diffexp}
H_a(r)-A_p(w;1,1-r)
=
\frac{b(c-b)}{2c^2}(p_0-p)r^2+O(r^3).
\end{equation}
More precisely, for each fixed choice of $a,b,c,p$, there exist constants
$\delta=\delta(a,b,c,p)>0$ and $C=C(a,b,c,p)>0$ such that
\begin{equation*}
\left|
H_a(r)-A_p(w;1,1-r)
-\frac{b(c-b)}{2c^2}(p_0-p)r^2
\right|
\leq Cr^3
\end{equation*}
for every $0<r<\delta$.
\end{lemma}

\begin{proof}
\textbf{Step 1:}
Claim that as $r\to 0^+,$
\begin{equation*}
H_a(r)=1-\frac bc r+\frac{(a-1)b(c-b)}{2c^2(c+1)}r^2+O(r^3).
\end{equation*}

Assume first that $a\neq 0$. The hypergeometric series gives
\begin{equation*}
\F(-a,b;c;r) =1-\frac{ab}{c}r +\frac{a(a-1)b(b+1)}{2c(c+1)}r^2 +O(r^3).
\end{equation*}
For $x\to0$,
\begin{equation*}
(1+x)^{\frac{1}{a}} =1+\frac{x}{a} +\frac{1}{2a}\left(\frac{1}{a}-1\right)x^2 +O(x^3).
\end{equation*}
Here
\begin{equation*}
x=-\frac{ab}{c}r +\frac{a(a-1)b(b+1)}{2c(c+1)}r^2 +O(r^3),
\end{equation*}
and hence
\begin{equation*}
x^2=\frac{a^2b^2}{c^2}r^2+O(r^3).
\end{equation*}
Substitution yields
\begin{equation*}
\begin{aligned}
H_a(r) &=1-\frac{b}{c}r +\frac{(a-1)b(b+1)}{2c(c+1)}r^2 +\frac{(1-a)b^2}{2c^2}r^2 +O(r^3)\\
&=1-\frac{b}{c}r +\frac{a-1}{2} \left[ \frac{b(b+1)}{c(c+1)}-\frac{b^2}{c^2} \right]r^2 +O(r^3).
\end{aligned}
\end{equation*}
The corresponding expansion for $H_a$ follows for $a\neq 0$.

Next considering $a=0$. Uniformly for $t\in[0,1]$ and $0<r\le\frac{1}{2}$,
\begin{equation*}
\log(1-rt)=-rt-\frac{r^2t^2}{2}+O(r^3).
\end{equation*}
After integration and use of \eqref{moments},
\begin{equation*}
\log H_0(r) =-\frac{b}{c}r -\frac{b(b+1)}{2c(c+1)}r^2 +O(r^3).
\end{equation*}
Writing the right-hand side as $y$, one has
\begin{equation*}
\mathrm{e}^y=1+y+\frac{y^2}{2}+O(r^3), \quad y^2=\frac{b^2}{c^2}r^2+O(r^3).
\end{equation*}
Therefore
\begin{equation*}
\begin{aligned}
H_0(r) &=1-\frac{b}{c}r +\frac{1}{2}\left( \frac{b^2}{c^2} -\frac{b(b+1)}{c(c+1)} \right)r^2 +O(r^3)\\
&=1-\frac{b}{c}r -\frac{b(c-b)}{2c^2(c+1)}r^2 +O(r^3),
\end{aligned}
\end{equation*}
which is the same expansion at $a=0$.

\textbf{Step 2:}
Claim that as $r\to 0^+$,
\begin{equation*}
A_p(w;1,1-r)= 1-\frac{b}{c}r +\frac{(p-1)b(c-b)}{2c^2}r^2 +O(r^3).
\end{equation*}

Fix now $p\neq 0$. Recall~\eqref{orders-1}, $1-w=\dfrac{b}{c}$,
\begin{equation*}
w+(1-w)(1-r)^p =1-\frac{bp}{c}r +\frac{bp(p-1)}{2c}r^2 +O(r^3).
\end{equation*}
For $x\to0$,
\begin{equation*}
(1+x)^{\frac{1}{p}} =1+\frac{x}{p} +\frac{1}{2p}\left(\frac{1}{p}-1\right)x^2 +O(x^3).
\end{equation*}
Here
\begin{equation*}
x=-\frac{bp}{c}r +\frac{bp(p-1)}{2c}r^2+O(r^3),
\end{equation*}
and hence
\begin{equation*}
x^2=\frac{b^2p^2}{c^2}r^2+O(r^3),
\end{equation*}
which yields 
\begin{equation*}
\begin{aligned}
A_p(w;1,1-r) &=1-\frac{b}{c}r +\frac{b(p-1)}{2c}r^2 +\frac{(1-p)b^2}{2c^2}r^2 +O(r^3)\\
&=1-\frac{b}{c}r +\frac{(p-1)b(c-b)}{2c^2}r^2 +O(r^3).
\end{aligned}
\end{equation*}
For $p=0$,
\begin{equation*}
\begin{aligned}
A_0(w;1,1-r) &=(1-r)^{\frac{b}{c}} =1-\frac{b}{c}r +\frac{b(b-c)}{2c^2}r^2 +O(r^3),
\end{aligned}
\end{equation*}
which is the same expansion at $p=0$.

Combining Step 1 and Step 2,
\begin{equation*}
\begin{aligned}
H_a(r)-A_p(w;1,1-r) &=\frac{b(c-b)}{2c^2} \left( \frac{a-1}{c+1}-(p-1) \right)r^2 +O(r^3)\\
&=\frac{b(c-b)}{2c^2}(p_0-p)r^2+O(r^3),
\end{aligned}
\end{equation*}
which proves \eqref{diffexp}.
\end{proof}

The coefficient in \eqref{diffexp} gives the necessary exponent restriction at $r=0$.

\begin{corollary} \label{localsharp}
If $A_\lambda(w;1,1-r)\leq H_a(r)$ for every sufficiently small $r>0$, then $\lambda\leq p_0$. If $H_a(r)\leq A_\mu(w;1,1-r)$ for every sufficiently small $r>0$, then $\mu\geq p_0$.
\end{corollary}
\begin{proof}
Suppose first that $\lambda>p_0$. By \eqref{diffexp},
\begin{equation*}
H_a(r)-A_\lambda(w;1,1-r) =-\frac{b(c-b)}{2c^2}(\lambda-p_0)r^2+O(r^3).
\end{equation*}
The coefficient of $r^2$ is strictly negative. Hence there exists $r_0>0$, depending on the fixed parameters $a,b,c,\lambda$, such that for $0<r<r_0$,
\begin{equation*}
H_a(r)-A_\lambda(w;1,1-r) \leq -\frac{b(c-b)}{4c^2}(\lambda-p_0)r^2<0,
\end{equation*}
which implies the first conclusion. The latter one follows similarly.
\end{proof}

The second sharp order is determined by the behavior at $r=1$.
By \eqref{moments},
\begin{equation*}
\int_0^1(1-t)\,\mathrm{d}\nu(t)
=
1-\frac{b}{c}
=
w.
\end{equation*}
For $a\neq0$, provided that $c+a-b>0$, direct evaluation gives
\begin{equation*}
\int_0^1(1-t)^a\,\mathrm{d}\nu(t)
=
\frac{B(b,c-b+a)}{B(b,c-b)}
=
\frac{\Gamma(c+a-b)\Gamma(c)}
{\Gamma(c-b)\Gamma(c+a)}
=
C_a.
\end{equation*}

We now compare the corresponding endpoint quantity with $w$. Since
$t\mapsto1-t$ is not constant $\nu$-almost everywhere, the relevant
forms of Jensen's inequality are strict. There are three cases.

\medskip
\noindent\emph{Case 1: $0<a<1$.}
The function $x\mapsto x^a$ is strictly concave on $(0,\infty)$.
Hence
\begin{equation*}
C_a
=
\int_0^1(1-t)^a\,\mathrm{d}\nu(t)
<
\left(\int_0^1(1-t)\,\mathrm{d}\nu(t)\right)^a
=
w^a.
\end{equation*}
Since $1/a>0$, taking the power $1/a$ preserves the inequality.

\medskip
\noindent\emph{Case 2: $a<0$.}
The function $x\mapsto x^a$ is strictly convex on $(0,\infty)$.
Therefore
\begin{equation*}
C_a
=
\int_0^1(1-t)^a\,\mathrm{d}\nu(t)
>
\left(\int_0^1(1-t)\,\mathrm{d}\nu(t)\right)^a
=
w^a.
\end{equation*}
Since $1/a<0$, taking the power $1/a$ reverses the inequality.

Thus, in both nonzero cases,
\begin{equation} \label{momentless}
C_a^{\frac{1}{a}}<w.
\end{equation}

\medskip
\noindent\emph{Case 3: $a=0$.}
Since $\log x$ is strictly concave on $(0,\infty)$, strict Jensen
inequality gives
\begin{equation*}
L_0
=
\int_0^1\log(1-t)\,\mathrm{d}\nu(t)
<
\log\left(\int_0^1(1-t)\,\mathrm{d}\nu(t)\right)
=
\log w.
\end{equation*}
Exponentiating yields
\begin{equation} \label{geomless}
\exp(L_0)<w.
\end{equation}

Consequently, for every $a<1$, the endpoint quantity associated with
$H_a$ is strictly smaller than $w$. The next lemma identifies these endpoint quantities as the limits of
$H_a$ and records the corresponding limits of the power means.

\begin{lemma} \label{endpoint}
Let $c>b>0$. If $a\neq 0$, assume $c+a-b>0$. Then
\begin{equation} \label{Hend}
\lim_{r\to1^-}H_a(r)=\begin{cases}C_a^{\frac{1}{a}}=w^{\frac{1}{p_E}},&a\neq 0,\\
\displaystyle\exp\left(\int_0^1\log(1-t) \mathrm{d}\nu(t)\right)=w^{\frac{1}{p_E}},&a=0,\end{cases}
\end{equation}
where $p_E$ is defined in \eqref{orders}, \eqref{pE0}. For $p>0$,
\begin{equation} \label{Aend}
\lim_{r\to1^-}A_p(w;1,1-r)=w^{\frac{1}{p}},
\end{equation}
while the limit is $0$ for $p\le0$.
\end{lemma}
\begin{proof}
If $a>0$, the integrand $(1-rt)^a$ in \eqref{euler} is bounded by $1$ and converges to $(1-t)^a$. If $a<0$, then $1-rt\geq 1-t$, so
\begin{equation*}
0<(1-rt)^a\le(1-t)^a.
\end{equation*}
The function $(1-t)^a$ is integrable with respect to $\nu$ precisely when $c+a-b>0$. Dominated convergence therefore gives the first line of \eqref{Hend} for $a\neq 0$.

For $a=0$, $\log(1-rt)\to\log(1-t)$. Since $1-rt\geq 1-t$ and both numbers lie in $(0,1]$,
\begin{equation*}
0\le-\log(1-rt)\le-\log(1-t).
\end{equation*}
The product of $-\log(1-t)$ with the beta density is $O((1-t)^{c-b-1}|\log(1-t)|)$ as $t\to1^-$ 
and $O(t^b)$ as $t\to0^+$.
Both bounds are integrable because $b>0$ and $c-b>0$.
Dominated convergence therefore proves the logarithmic case.

Now we consider $\lim_{r\to1^-}A_p(w;1,1-r)$. For $p>0$, \eqref{Aend} follows from \eqref{mean} directly.
If $p=0$, then $A_0=(1-r)^{1-w}\to0$ because $1-w=\frac{b}{c}>0$. 
If $p<0$, factor $(1-r)^p$ inside the bracket:
\begin{equation*}
A_p(w;1,1-r)=(1-r)\left[w(1-r)^{-p}+1-w\right]^{\frac{1}{p}},
\end{equation*}
which implies $A_p\to0$.
\end{proof}

Comparison of these endpoint values with the power mean gives the necessary restriction at $p_E$.
Notice that under the assumptions of Lemma \ref{endpoint} , the endpoint order $p_E$ is strictly positive.
\begin{corollary} \label{endpointsharp}
With the same hypotheses as in Lemma \ref{endpoint},
\begin{enumerate}
\item If
\begin{equation*}
A_\lambda(w;1,1-r)\leq H_a(r)\quad(0<r<1),
\end{equation*}
then $\lambda\leq p_E$. 
\item If
\begin{equation*}
H_a(r)\leq A_\mu(w;1,1-r)\quad(0<r<1),
\end{equation*}
then $\mu\geq p_E$. 
\end{enumerate}
\end{corollary}
\begin{proof}
If $\lambda\le0$, then $\lambda\leq p_E$. For $\lambda>0$, passage to $r\to1^-$ gives
$w^{\frac{1}{\lambda}}\leq w^{\frac{1}{p_E}}.$
Since
\begin{equation*}
\frac{\mathrm{d}}{\mathrm{d}p}w^{\frac{1}{p}} =w^{\frac{1}{p}}\frac{-\log w}{p^2}>0\quad(p>0),
\end{equation*}
one has $\lambda\leq p_E$.

If $\mu\le0$, then Lemma \ref{endpoint} gives $A_\mu(w;1,1-r)\to0$, whereas $H_a(r)\to w^{\frac{1}{p_E}}>0$. Hence $\mu>0$, and passage to $r\to1^-$ gives
\begin{equation*}
w^{\frac{1}{p_E}}\leq w^{\frac{1}{\mu}},
\end{equation*}
so $\mu\geq p_E$.
\end{proof}

\subsection{The differential equation and the residual}
The endpoint analysis gives only the necessary restrictions. We now begin the sufficiency argument by rewriting the beta representation as a second-order equation and inserting the candidate power mean into the same operator. The resulting scalar residual $S_p$ contains all sign information needed for the global comparison.

Set $v=1-2t$,
the measure \eqref{beta} becomes $\rho(v) \mathrm{d}v$ on $(-1,1)$, where
\begin{equation*}
\rho(v)=\frac{(1-v)^{b-1}(1+v)^{c-b-1}}{2^{c-1}B(b,c-b)}.
\end{equation*}
Considering $r=1-\mathrm{e}^{-2u}$, for $a\neq 0$ we define
\begin{equation} \label{f}
f(u):=\mathrm{e}^{au}H_a(1-\mathrm{e}^{-2u})^a=\int_{-1}^1z(u,v)^a\rho(v) \mathrm{d}v,
\end{equation}
where
\begin{equation*}
z(u,v):=\cosh u+v\sinh u.
\end{equation*}
Since $-1<v<1$ and $u>0$,
$\mathrm{e}^{-u}<z(u,v)<\mathrm{e}^u.$
Fix $U>0$. For $0\leq u\leq U$,
$$
\mathrm{e}^{-U}\leq z(u,v)\leq \mathrm{e}^U,$$ and $$ |z_u(u,v)|+|z_{uu}(u,v)|\le2\mathrm{e}^U.$$
 
The first two $u$-derivatives of $z^a$ are uniformly dominated on $[0,U]\times(-1,1)$, so differentiation under the integral sign is valid.

The Jacobi-type density identity for $\rho$ gives the differential equation satisfied by the transformed hypergeometric function.

\begin{lemma}
Set $d=c-2b$. For $a\neq 0$, the function $f$ satisfies
\begin{equation} \label{ode}
f''+(c\coth u+d)f'-a(a+c+d\coth u)f=0,
\end{equation}
with
\begin{equation} \label{initial}
f(0)=1,\quad f'(0)=\frac{ad}{c}.
\end{equation}
\end{lemma}
\begin{proof}
Differentiation of the density gives
\begin{equation} \label{rho}
\frac{\mathrm{d}}{\mathrm{d}v}\left[(1-v^2)\rho(v)\right]=(d-cv)\rho(v).
\end{equation}
For fixed $u>0$, the quantity $z(u,v)$ is bounded above and below by positive constants for $-1\leq v\le1$. Moreover,
$(1-v^2)\rho(v)=O((1-v)^b),(v\to1^-),$
and
$(1-v^2)\rho(v)=O((1+v)^{c-b}),(v\to-1^+).$
Since $b>0$ and $c-b>0$,
$(1-v^2)\rho(v)z(u,v)^{a-1}\to0 \quad(v\to\pm1).$
Therefore
\begin{equation*}
\int_{-1}^1 \frac{\mathrm{d}}{\mathrm{d}v} \left[(1-v^2)\rho(v)z(u,v)^{a-1}\right] \mathrm{d}v=0,
\end{equation*}
combining with \eqref{rho}, which yields
\begin{equation} \label{ibp}
c\int_{-1}^1vz^{a-1}\rho(v) \mathrm{d}v
=d\int_{-1}^1z^{a-1}\rho(v) \mathrm{d}v+(a-1)\sinh u\int_{-1}^1(1-v^2)z^{a-2}\rho(v) \mathrm{d}v.
\end{equation}

Next,
\begin{equation} \label{zu}
z_u=\sinh u+v\cosh u,\quad z_{uu}=z,
\end{equation}
and direct expansion gives
$z^2-z_u^2=(1-v^2)(\cosh^2u-\sinh^2u) =1-v^2.$
Thus
\begin{equation*}
z_u^2=z^2-(1-v^2).
\end{equation*}
Recall \eqref{f}, differentiating $f=\int z^a\rho \mathrm{d}v$ twice gives
\begin{equation*}
f'=a\int_{-1}^1z^{a-1}z_u\rho(v) \mathrm{d}v
\end{equation*}
and
\begin{equation*}
f''=a(a-1)\int_{-1}^1z^{a-2}z_u^2\rho(v) \mathrm{d}v +a\int_{-1}^1z^{a-1}z_{uu}\rho(v) \mathrm{d}v.
\end{equation*}
Since $z_{uu}=z$ and $z_u^2=z^2-(1-v^2)$,
\begin{equation} \label{fpp}
f''=a^2f-a(a-1)\int_{-1}^1(1-v^2)z^{a-2}\rho(v) \mathrm{d}v.
\end{equation}

The pointwise identity
\begin{equation}\label{pointwise id1}
z_u\coth u-z
=\frac{\cosh u}{\sinh u}(\sinh u+v\cosh u) -(\cosh u+v\sinh u)
=\frac{v}{\sinh u}
\end{equation}
implies
\begin{equation*}
\frac{1}{\sinh u}\int_{-1}^1vz^{a-1}\rho(v) \mathrm{d}v =\frac{\coth u}{a}f'-f.
\end{equation*}
Similarly, $z\cosh u-z_u\sinh u =\cosh^2u-\sinh^2u=1,$ and hence
\begin{equation}\label{pointwise id2}
\dfrac{1}{\sinh u}=z\coth u-z_u.
\end{equation} 
Therefore
\begin{equation*}
\frac{1}{\sinh u}\int_{-1}^1z^{a-1}\rho(v) \mathrm{d}v =\coth u f-\frac{1}{a}f'.
\end{equation*}

Divide \eqref{ibp} by $\sinh u$. Then
\begin{equation*}
(a-1)\int_{-1}^1(1-v^2)z^{a-2}\rho(v) \mathrm{d}v
= \frac{c}{\sinh u}\int_{-1}^1vz^{a-1}\rho(v) \mathrm{d}v
- \frac{d}{\sinh u}\int_{-1}^1z^{a-1}\rho(v) \mathrm{d}v.
\end{equation*}
Substituting this into \eqref{fpp} gives
\begin{equation*}
\begin{aligned}
f'' &=a^2f -ac\left(\frac{\coth u}{a}f'-f\right) +ad\left(\coth u f-\frac{1}{a} f'\right)\\
&=a^2f-c\coth u f'+ac f+ad\coth u f-df',
\end{aligned}
\end{equation*}
which proves \eqref{ode}.

Finally, $z(0,v)=1$ and $z_u(0,v)=v$, so
\begin{equation*}
f(0)=\int_{-1}^1\rho(v) \mathrm{d}v=1 \quad f'(0)=a\int_{-1}^1v\rho(v) \mathrm{d}v.
\end{equation*}
Recalling $v=1-2t$ and \eqref{moments}, we obtain
\begin{equation*}
f'(0)=a\int_{0}^1 (1-2t) \mathrm{d} \nu(t) = a(1-\frac{2b}{c}) = \frac{ad}{c},
\end{equation*}
which proves \eqref{initial}.
\end{proof}

The preceding lemma identifies the differential operator governing the
transformed hypergeometric function. To compare $H_a$ with the power
means, we now express $A_p$ in the same variable $u$ and construct the
corresponding comparison function.

For $p>0$, put
$v_0=\dfrac{1}{2}\log\dfrac{w}{1-w}=\dfrac{1}{2}\log\dfrac{c-b}{b}.$
Then
\begin{equation} \label{tanh}
\tanh v_0=\frac{c-2b}{c}=\frac{d}{c}.
\end{equation}
From the definition of $v_0$,
\begin{equation*}
w=\frac{\mathrm{e}^{v_0}}{2\cosh v_0},\quad 1-w=\frac{\mathrm{e}^{-v_0}}{2\cosh v_0},
\end{equation*}
and hence $w\mathrm{e}^{pu}+(1-w)\mathrm{e}^{-pu}=\dfrac{\cosh(pu+v_0)}{\cosh v_0}$. 
Now define
\begin{equation} \label{g}
\begin{aligned}
g(u)&=[w\mathrm{e}^{pu}+(1-w)\mathrm{e}^{-pu}]^{\frac{a}{p}}=\left[\frac{\cosh(pu+v_0)}{\cosh v_0}\right]^{\frac{a}{p}}.
\end{aligned}
\end{equation}
Since $1-r=\mathrm{e}^{-2u}$,
\begin{equation*}
A_p(w;1,1-r)=A_p(w;1,\mathrm{e}^{-2u})=\mathrm{e}^{-u}[w\mathrm{e}^{pu}+(1-w)\mathrm{e}^{-pu}]^{\frac{1}{p}}.
\end{equation*}
Together with \eqref{f}, this gives
\begin{equation} \label{ratio}
\left(\frac{f(u)}{g(u)}\right)^{\frac{1}{a}}=\frac{H_a(r)}{A_p(w;1,1-r)}.
\end{equation}
Also $g(0)=1$ and, by \eqref{tanh},
\begin{equation*}
g'(0)=a\tanh v_0=\frac{ad}{c}=f'(0).
\end{equation*}
Hence $f(0)=g(0)$ and $f'(0)=g'(0)$.

Let
\begin{equation*}
\mathcal{L}y=y''+(c\coth u+d)y'-a(a+c+d\coth u)y.
\end{equation*}
Then \eqref{ode} says $\mathcal{L}f=0$.

The following lemma computes the residual of the comparison function $g$.

\begin{lemma} \label{reslem}
Put $k=2p-1$. Then
\begin{equation} \label{res}
\mathcal{L}g=\frac{ag}{\cosh^2(pu+v_0)}S_p(u),
\end{equation}
where
\begin{equation} \label{S}
S_p(u)=p-a+\frac{c}{2}\left(\frac{\sinh(ku)}{\sinh u}-1\right)-\frac{d}{2}\frac{\cosh u-\cosh(ku)}{\sinh u}.
\end{equation}
Moreover,
\begin{equation} \label{S0}
S_p(0^+)=p-a+c(p-1)=(1+c)(p-p_0).
\end{equation}
\end{lemma}
\begin{proof}
Logarithmic differentiation of \eqref{g} gives
\begin{equation*}
\frac{g'}{g}=a\tanh(pu+v_0).
\end{equation*}
Differentiating once more,
we obtain
\begin{equation*}
\frac{g''}{g} =\left(\frac{g'}{g}\right)^2 +\left(\frac{g'}{g}\right)' =a^2\tanh^2(pu+v_0)+ap\operatorname{sech}^2(pu+v_0).
\end{equation*}
Substitution into the definition of $\mathcal{L}$ yields
\begin{equation*}
\begin{aligned}
\frac{\mathcal{L}g}{ag} &=a\tanh^2(pu+v_0) +p\operatorname{sech}^2(pu+v_0)\\
&+(c\coth u+d)\tanh(pu+v_0) -(a+c+d\coth u).
\end{aligned}
\end{equation*}
Using $\tanh^2x=1-\operatorname{sech}^2x$,
multiplication by $\cosh^2(pu+v_0)$ gives
\begin{equation} \label{rawS}
\begin{aligned}
\frac{\cosh^2(pu+v_0)}{ag}\mathcal{L}g &=p-a+c\cosh^2(pu+v_0)(\coth u\tanh(pu+v_0)-1)\\
&+d\cosh^2(pu+v_0)(\tanh(pu+v_0)-\coth u).
\end{aligned}
\end{equation}

For the first hyperbolic term,
\begin{equation*}
\begin{aligned}
\cosh^2(pu+v_0)&(\coth u\tanh(pu+v_0)-1)\\
&= \frac{\cosh u\sinh(pu+v_0)\cosh(pu+v_0) -\sinh u\cosh^2(pu+v_0)}{\sinh u}\\
&= \frac{\sinh((2p-1)u+2v_0)}{2\sinh u}-\frac{1}{2}.
\end{aligned}
\end{equation*}
Similarly,
\begin{equation*}
\begin{aligned}
\cosh^2(pu+v_0)&(\tanh(pu+v_0)-\coth u)\\
&= \frac{\sinh u\sinh(pu+v_0)\cosh(pu+v_0) -\cosh u\cosh^2(pu+v_0)}{\sinh u}\\
&= -\frac{\cosh((2p-1)u+2v_0)}{2\sinh u} -\frac{1}{2}\coth u.
\end{aligned}
\end{equation*}

Recall \eqref{tanh}:
$ c\sinh v_0=d\cosh v_0.$
Combine this with the double-angle formulas:
\begin{equation*}
c\cosh(2v_0)-d\sinh(2v_0)=c, \quad
c\sinh(2v_0)-d\cosh(2v_0)=d.
\end{equation*}
Hence, with $k=2p-1$,
\begin{equation*}
\begin{aligned}
c\sinh(ku+2v_0)&-d\cosh(ku+2v_0)\\
&= [c\cosh(2v_0)-d\sinh(2v_0)]\sinh(ku)\\
&\quad+ [c\sinh(2v_0)-d\cosh(2v_0)]\cosh(ku)\\
&=c\sinh(ku)+d\cosh(ku).
\end{aligned}
\end{equation*}
Consequently the last two terms in \eqref{rawS} equal
\begin{equation*}
\begin{aligned}
\frac{c\sinh(ku)+d\cosh(ku)}{2\sinh u}& -\frac{c}{2}-\frac{d}{2}\coth u\\
&= \frac{c}{2}\left(\frac{\sinh(ku)}{\sinh u}-1\right) -\frac{d}{2}\frac{\cosh u-\cosh(ku)}{\sinh u}.
\end{aligned}
\end{equation*}
Adding the remaining term $p-a$ proves \eqref{res} and \eqref{S}.

Finally, as $u\to 0^{+}$,
$$
\dfrac{\sinh(ku)}{\sinh u}\to k 
,\qquad
\dfrac{\cosh u-\cosh(ku)}{\sinh u} \to 0,
$$
therefore
\begin{equation*}
S_p(0^+) =p-a+\frac{c}{2}(k-1) =p-a+c(p-1) =(1+c)(p-p_0),
\end{equation*}
which proves \eqref{S0}.
\end{proof}

\begin{lemma} \label{hyp}
Fix $k>-1$. For $u>0$, recalling~\eqref{S}, naturally we define
\begin{equation}\label{FkQk}
F_k(u)=\frac{\sinh(ku)}{\sinh u},
\qquad
Q_k(u)=\frac{\cosh u-\cosh(ku)}{\sinh u}.
\end{equation}
Then the following assertions hold.
\begin{enumerate}
\item The function $F_k$ is strictly increasing on $u\in(0,\infty)$
if $-1<k<0$ or $k>1$, and strictly decreasing on $u\in(0,\infty)$
if $0<k<1$. At the remaining values $k=0$ and $k=1$,
one has $F_0\equiv0$ and $F_1\equiv1$, respectively.

\item If $-1<k<1$, then $Q_k$ is strictly increasing
on $(0,\infty)$.
\end{enumerate}
\end{lemma}
This elementary lemma will be used in Sections \ref{sec-3} and \ref{sec-4} to determine the sign and monotonicity of the residual $S_p$. Its proof is deferred to Appendix A.

\subsection{Weighted-Wronskian comparison}
The following weighted-Wronskian criterion converts residual sign
information into a pointwise comparison of two positive functions.

\begin{proposition}[weighted-Wronskian comparison] \label{comp}
Let $P,V\in C((0,\infty);\mathbb{R})$, and define
the second-order linear differential operator
\begin{equation*}
\mathcal{A}y=y''+P(u)y'-V(u)y.
\end{equation*}
Let $m\in C^1((0,\infty))$ be strictly positive and satisfy
$m'=Pm$. Suppose that $\tilde f,\tilde g>0$ are
$C^2$ functions on $(0,\infty)$ such that
\begin{equation*}
\mathcal{A}\tilde f=0,\quad \mathcal{A}\tilde g=\sigma(u)\tilde g.
\end{equation*}
Set
\begin{equation*}
R=\frac{\tilde f}{\tilde g},\quad
W=m(\tilde f'\tilde g-\tilde f\tilde g'),
\end{equation*}
and assume $ R(0^+)=1$ and $ W(0^+)=0$. Then:
\begin{enumerate}
  \item If $\sigma<0$ on $(0,\infty)$, then $ R>1$ and $ R'>0$;
  \item If $\sigma>0$ on $(0,\infty)$, then $ R<1$ and $ R'<0$.
\end{enumerate}

Assume in addition that $ R(\infty)=1$. If there exists $u_0>0$ such that
\begin{equation*}
\sigma(u)<0\quad(0<u<u_0),\quad \sigma(u)>0\quad(u>u_0),
\end{equation*}
then $ R(u)>1$ for every $u>0$. If the two signs are reversed, then $ R(u)<1$ for every $u>0$.
\end{proposition}
Its proof is deferred to Appendix A.

In the present problem, set $\tilde{f}=f,\tilde{g}=g,m(u)=(\sinh u)^c\mathrm{e}^{du}$,then
\begin{equation*}
 R(u)=\frac{f(u)}{g(u)},\quad W(u)=m(u)[f'(u)g(u)-f(u)g'(u)].
\end{equation*}
Since $\dfrac{m'}{m}=c\coth u+d$, Proposition \ref{comp}, \eqref{ode} and \eqref{res} give
\begin{equation} \label{Wprime}
W'=-\frac{amfg}{\cosh^2(pu+v_0)}S_p(u),
\end{equation}
and
\begin{equation} \label{Rprime}
R'(u)=\frac{W(u)}{m(u)g(u)^2}.
\end{equation}
The initial data imply $W(0^+)=0$ because $f(0)=g(0)$, $f'(0)=g'(0)$ and $m(u)\sim u^c$ as $u\to0^+$.

\begin{corollary} \label{direct}
Assume $a\neq0$ and $p>0$.
\begin{enumerate}
\item If $S_p(u)<0$ for every $u>0$, then
\begin{equation*}
\frac{H_a(r)}{A_p(w;1,1-r)}>1
\end{equation*}
for every $r\in(0,1)$, and this quotient is strictly increasing in $r$. If $S_p(u)>0$ for every $u>0$, then the quotient is strictly smaller than $1$ and strictly decreasing.

\item Suppose, in addition, that $p=p_E$. Then
\begin{equation} \label{equalends}
\lim_{u\to0^+}\frac{f(u)}{g(u)}=1,\quad
\lim_{u\to\infty}\frac{f(u)}{g(u)}=1.
\end{equation}
If $aS_{p_E}$ is strictly increasing, then $f/g>1$ on $(0,\infty)$; if $aS_{p_E}$ is strictly decreasing, then $f/g<1$ on $(0,\infty)$. If $S_{p_E}$ is constant, then $S_{p_E}\equiv0$ and $f/g\equiv1$.
\end{enumerate}
\end{corollary}
\begin{proof}
By \eqref{res}, the residual in Proposition \ref{comp} is
\begin{equation*}
\sigma_p(u)=\frac{aS_p(u)}{\cosh^2(pu+v_0)}.
\end{equation*}
For part (1), if $S_p<0$, Proposition \ref{comp} gives $\operatorname{sgn}R'=\operatorname{sgn}a$. Hence
\begin{equation*}
\frac{\mathrm{d}}{\mathrm{d}u}R(u)^{\frac{1}{a}}=\frac{1}{a}R(u)^{\frac{1}{a}-1}R'(u)>0.
\end{equation*}
Since $R(0^+)=1$, \eqref{ratio} gives $H_a/A_p>1$. If $S_p>0$, the same argument gives $\operatorname{sgn}R'=-\operatorname{sgn}a$ and therefore $H_a/A_p<1$. Finally, $\frac{\mathrm{d}r}{\mathrm{d}u}=2\mathrm{e}^{-2u}>0$, so the monotonicity in $u$ and $r$ is the same.

For part (2), the first limit in \eqref{equalends} follows from $f(0)=g(0)=1$. By \eqref{Hend},
\begin{equation*}
\mathrm{e}^{-au}f(u)=H_a(1-\mathrm{e}^{-2u})^a\to C_a,
\end{equation*}
whereas
\begin{equation*}
\mathrm{e}^{-au}g(u)=[w+(1-w)\mathrm{e}^{-2p_Eu}]^{\frac{a}{p_E}}\to w^{\frac{a}{p_E}}=C_a.
\end{equation*}
Thus $R(\infty)=1$.

Assume that $aS_{p_E}$ is strictly increasing. It cannot be nonnegative on all of $(0,\infty)$. 
Otherwise, \eqref{Wprime} and \eqref{Rprime} would make $R$ nonincreasing with equal endpoint values, forcing $R\equiv1$ and hence $aS_{p_E}\equiv0$, a contradiction. The case $aS_{p_E}\leq0$ is identical. Hence $aS_{p_E}$ takes both signs, and strict increase gives a unique $u_0$ at which its sign changes from negative to positive. Since $\cosh^2(p_Eu+v_0)>0$, the residual $\sigma_{p_E}=aS_{p_E}/\cosh^2(p_Eu+v_0)$ has the same sign pattern. Proposition \ref{comp} gives $R>1$.

If $aS_{p_E}$ is strictly decreasing, the same argument and Proposition \ref{comp} yields $R<1$. Finally, if $S_{p_E}$ is constant and nonzero, \eqref{Wprime} would make $R$ strictly monotone, contradicting \eqref{equalends}. Hence $S_{p_E}\equiv0$, and \eqref{Wprime}--\eqref{Rprime} give $R\equiv1$.
\end{proof}

\subsection{The logarithmic analogue}

The preceding comparison is multiplicative and applies when $a\neq0$. To complete the sufficiency mechanism for $a=0$, we construct its additive logarithmic analogue. Put
\begin{equation*}
\ell(u)=\int_{-1}^1\log z(u,v)\rho(v) \mathrm{d}v.
\end{equation*}
Fix $U>0$. On $0\leq u\leq U$ we have $z\geq \mathrm{e}^{-U}$ and $|z_u|+|z_{uu}|\le2\mathrm{e}^U$. Hence
\begin{equation*}
\left|\frac{z_u}{z}\right|\le2\mathrm{e}^{2U},\quad \left|\frac{z_{uu}}{z}-\frac{z_u^2}{z^2}\right|\le2\mathrm{e}^{2U}+4\mathrm{e}^{4U}.
\end{equation*}
These bounds are uniform in $v\in(-1,1)$, so $\ell$ may be differentiated twice on compact $u$-intervals. Since $z_{uu}=z$ and \eqref{zu} holds,
\begin{equation*}
\ell''=\int_{-1}^1\left(\frac{z_{uu}}{z}-\frac{z_u^2}{z^2}\right)\rho(v) \mathrm{d}v =\int_{-1}^1\frac{1-v^2}{z^2}\rho(v) \mathrm{d}v.
\end{equation*}
Note that \eqref{ibp} is valid for every real exponent in $z^{a-1}$. Setting $a=0$ therefore gives
\begin{equation*}
c\int_{-1}^1\frac{v}{z}\rho(v) \mathrm{d}v =d\int_{-1}^1\frac{1}{z}\rho(v) \mathrm{d}v -\sinh u\int_{-1}^1\frac{1-v^2}{z^2}\rho(v) \mathrm{d}v.
\end{equation*}
The two pointwise identities \eqref{pointwise id1}, \eqref{pointwise id2} give
\begin{equation*}
\frac{1}{\sinh u}\int_{-1}^1\frac{v}{z}\rho(v) \mathrm{d}v=\ell'\coth u-1,
\end{equation*}
and
\begin{equation*}
\frac{1}{\sinh u}\int_{-1}^1\frac{1}{z}\rho(v) \mathrm{d}v=\coth u-\ell'.
\end{equation*}
Substitution into the preceding identity yields
\begin{equation} \label{ellode}
\ell''+(c\coth u+d)\ell'=c+d\coth u.
\end{equation}
For $a=0$, let $S_p$ denote the expression in \eqref{S}. For $p>0$, differentiation gives
\begin{equation*}
\frac{\mathrm{d}}{\mathrm{d}u} \left[ \frac{1}{p}\log\frac{\cosh(pu+v_0)}{\cosh v_0} \right] =\tanh(pu+v_0)
\end{equation*}
and
\begin{equation*}
\frac{\mathrm{d}^2}{\mathrm{d}u^2} \left[ \frac{1}{p}\log\frac{\cosh(pu+v_0)}{\cosh v_0} \right] =p\operatorname{sech}^2(pu+v_0).
\end{equation*}
The residual obtained by substituting this comparison function
into \eqref{ellode} is
\begin{equation*}
p\operatorname{sech}^2(pu+v_0) +(c\coth u+d)\tanh(pu+v_0) -(c+d\coth u).
\end{equation*}
Multiplying by $\cosh^2(pu+v_0)$ and using the two hyperbolic identities established in the proof of Lemma \ref{reslem}, one obtains
\begin{equation*}
\begin{aligned}
\cosh^2(pu+v_0)& \left[ p\operatorname{sech}^2(pu+v_0) +(c\coth u+d)\tanh(pu+v_0) -(c+d\coth u) \right]\\
&= p+\frac{c}{2}\left(\frac{\sinh((2p-1)u)}{\sinh u}-1\right) -\frac{d}{2}\frac{\cosh u-\cosh((2p-1)u)}{\sinh u}.
\end{aligned}
\end{equation*}
The right-hand side is exactly $S_p(u)$ with $a=0$. Hence
\begin{equation} \label{hres}
p\operatorname{sech}^2(pu+v_0) +(c\coth u+d)\tanh(pu+v_0)-(c+d\coth u)
=\frac{S_p(u)}{\cosh^2(pu+v_0)}.
\end{equation}
Also,
\begin{equation*}
\ell(u)=u+\log H_0(1-\mathrm{e}^{-2u}),
\end{equation*}
and
\begin{equation*}
	\frac{1}{p}\log\frac{\cosh(pu+v_0)}{\cosh v_0} =u+\log A_p(w;1,\mathrm{e}^{-2u}).
\end{equation*} 
For $p>0$, define the logarithmic comparison function
\begin{equation}\label{Dquot}
D_p(u):=
\ell(u)-\frac{1}{p}\log\frac{\cosh(pu+v_0)}{\cosh v_0}
=\log\frac{H_0(r)}{A_p(w;1,1-r)}.
\end{equation}

Thus the sign of $D_p$ determines the comparison between
$H_0$ and $A_p$.
Subtracting \eqref{hres} from \eqref{ellode} and multiplying by $m$ gives
\begin{equation} \label{JD}
\left(mD_p'\right)'=-\frac{mS_p(u)}{\cosh^2(pu+v_0)}.
\end{equation}
At $u=0$,
\begin{equation*}
\ell(0)=0,\quad \ell'(0)=\int_{-1}^1v\rho(v) \mathrm{d}v=\frac{d}{c},\quad \tanh v_0=\frac{d}{c}.
\end{equation*}
Hence $D_p(0^+)=D_p'(0^+)=0$. Since $m(u)\sim u^c$ and $c>0$, also
\begin{equation*}
m(u)D_p'(u)\to0\quad(u\to0^+).
\end{equation*}

The following lemma is the logarithmic analogue of Corollary \ref{direct}.

\begin{lemma} \label{logturning}
Assume $p=p_E>0$. Then $D_p(0^+)=D_p(\infty)=0$.
\begin{enumerate}
\item If $S_p$ is strictly increasing, then $D_p>0$ on $(0,\infty)$.
\item If $S_p$ is strictly decreasing, then $D_p<0$ on $(0,\infty)$.
\item If $S_p$ is constant, then $S_p\equiv0$ and $D_p\equiv0$.
\end{enumerate}
\end{lemma}

\begin{proof}
By \eqref{Hend} and \eqref{Dquot},
$D_p(0^+)=D_p(\infty)=0$.
Set $J=mD_p'$ then \eqref{JD} gives
\begin{equation*}
J'=-\frac{m}{\cosh^2(pu+v_0)}S_p 
\end{equation*}
with  $J(0^+)=0.$
Suppose first that $S_p$ is strictly increasing.
If $S_p\geq0$ throughout $(0,\infty)$, then $J\leq0$
and hence $D_p'\leq0$.
Similarly, if $S_p\leq0$ throughout $(0,\infty)$,
then $D_p'\geq0$.
In either case, the equal endpoint values force $D_p\equiv0$,
and consequently $S_p\equiv0$, a contradiction.
Thus there is a unique $u_0>0$ such that
\begin{equation*}
S_p(u)<0\quad(0<u<u_0),\qquad
S_p(u)>0\quad(u>u_0).
\end{equation*}

It follows that $J$ is strictly increasing on $(0,u_0)$
and strictly decreasing on $(u_0,\infty)$.
Since $J(0^+)=0$, we have $J>0$ on $(0,u_0]$.
$J$ has a unique zero $u_1>u_0$,
otherwise $J>0$ throughout $(0,\infty)$, and hence $D_p'>0$ there.
This would contradict $D_p(0^+)=D_p(\infty)=0$.
Therefore
\begin{equation*}
D_p'(u)>0\quad(0<u<u_1),\qquad
D_p'(u)<0\quad(u>u_1).
\end{equation*}
Together with $D_p(0^+)=D_p(\infty)=0$,
this proves $D_p(u)>0$ for every $u>0$.

If $S_p$ is strictly decreasing, apply the preceding argument
to $-D_p$, $-J$ and $-S_p$ to obtain $D_p<0$.

Finally, suppose that $S_p$ is constant.
If it were nonzero, then $D_p$ strictly monotone,
contradicting its equal endpoint values.
Hence $S_p\equiv0$, which gives $J\equiv0$ and $D_p\equiv0$.
\end{proof}

\section{The range \texorpdfstring{$a\leq1$}{$a\leq1$}} \label{sec-3}
Corollaries \ref{localsharp} and \ref{endpointsharp} already give the necessary restrictions on the admissible lower and upper orders. It remains to establish the sharp comparisons at $p_0$ and $p_E$, their order relations, and the exceptional identity; Lemma \ref{meanmono} then yields the sufficiency of the resulting parameter restrictions. We separate the nonzero exponent from the logarithmic case because the former is multiplicative, while the latter is governed by the additive comparison in Section \ref{sec-2}.

\begin{proposition} \label{prop-nonzero}
Assume $a<1$, $a\neq 0$, $c>b>0$ and $c>b-a$.
\begin{enumerate}
\item If $c\geq \max(1-2a,2b)$ holds, then $0<p_0\leq p_E<1$. If $c=2b=1-2a$, then $p_0=p_E=\frac{1}{2}$ and \eqref{identity} holds. Otherwise $p_0<p_E$ and \eqref{upsharp} holds.
\item If $c\le\min(1-2a,2b)$ holds, then $0<p_E\leq p_0<1$. The equality case is again exactly $c=2b=1-2a$. Outside this case, \eqref{lowsharp} holds.
\end{enumerate}
\end{proposition}
\begin{proof}
By \eqref{momentless},
\begin{equation*}
0<C_a^{\frac{1}{a}}<w<1.
\end{equation*}
Since $C_a^{\frac{1}{a}}=w^{\frac{1}{p_E}}$ and $0<w<1$, we obtain $0<p_E<1$. Moreover, $c>b-a$ gives $p_0>0$.

\textbf{Step 1: Upper branch}
\quad

Assume $c\geq \max(1-2a,2b)$. Then $d=c-2b\geq0$, 
and $c\geq 1-2a$ is equivalent to $p_0\geq \frac12$.
Since $a<1$, one also has $p_0<1$. In conclusion, if we put $k_0=2p_0-1$, then $0\leq k_0<1$. By \eqref{S0}, $S_{p_0}(0^+)=0$. Recalling \eqref{FkQk}, we have $F_{k_0}(0^+)=k_0$ and $Q_{k_0}(0^+)=0$, and formula \eqref{S} gives
\begin{equation} \label{Sdiff}
S_{p_0}(u)=\frac{c}{2}[F_{k_0}(u)-k_0]-\frac{d}{2}Q_{k_0}(u).
\end{equation}
If $0<k_0<1$, Lemma \ref{hyp} gives $F_{k_0}(u)<k_0$ and $Q_{k_0}(u)>0$. 
Thus $S_{p_0}(u)<0$ for all $u>0$;
If $k_0=0$ and $d>0$, then $F_0\equiv0$ and $Q_0(u)>0$, so the same conclusion holds. 
Thus $S_{p_0}$ can vanish identically only if
\begin{equation*}
k_0=0,\quad d=0.
\end{equation*}
The first equality is equivalent to $c=1-2a$, while the second is equivalent to $c=2b$. 
Hence the exceptional case is
\begin{equation} \label{exception}
c=2b=1-2a.
\end{equation}

\textbf{Case 1:} Outside \eqref{exception}, Corollary \ref{direct} gives
\begin{equation} \label{p0lower}
A_{p_0}(w;1,1-r)<H_a(r),\quad 0<r<1,
\end{equation}
and $H_a/A_{p_0}$ is strictly increasing. Therefore, by lemma \ref{endpoint}, for any $r_*\in(0,1)$,
\begin{equation*}
1<\frac{H_a(r_*)}{A_{p_0}(w;1,1-r_*)}\leq\lim_{r\to1^-}\frac{H_a(r)}{A_{p_0}(w;1,1-r)}=\frac{w^{\frac{1}{p_E}}}{w^{\frac{1}{p_0}}}.
\end{equation*}
Since $p\mapsto w^{\frac{1}{p}}$ is strictly increasing on $(0,\infty)$,
\begin{equation} \label{orderup}
p_0<p_E<1.
\end{equation}
By \eqref{orderup}, $0<k_E:=2p_E-1<1$, and Lemma \ref{hyp} together with \eqref{S} gives
\begin{equation*}
F_{k_E}'<0,\quad Q_{k_E}'>0,\quad S_{p_E}'<0.
\end{equation*}
If $a>0$, then $aS_{p_E}$ is strictly decreasing, so Corollary \ref{direct} gives $f/g<1$. Since $\frac{1}{a}>0$, \eqref{ratio} yields $H_a/A_{p_E}<1$. If $a<0$, similar argument again gives $H_a/A_{p_E}<1$. Thus
\begin{equation} \label{pEupper}
H_a(r)<A_{p_E}(w;1,1-r),\quad 0<r<1.
\end{equation}
Combining \eqref{p0lower} and \eqref{pEupper} proves \eqref{upsharp}.

\textbf{Case 2:} In the exceptional case \eqref{exception}, $p_0=\frac{1}{2}$, $d=0$ and $k_0=0$. Hence $S_{p_0}\equiv0$. Equations \eqref{Wprime} and \eqref{Rprime} give $W\equiv0$ and $f/g\equiv1$, so \eqref{identity} holds and $p_E=p_0=\frac{1}{2}$.

\medskip
\noindent\textbf{Step 2: Lower branch.}

Assume $c\le\min(1-2a,2b)$. Then $d\leq0$ and
$0<p_0\leq\frac12$, so $-1<k_0=2p_0-1\leq0$.
By Lemma \ref{hyp} and \eqref{Sdiff}, we have
$S_{p_0}(u)>0$ for every $u>0$, unless $k_0=d=0$.
The latter is precisely the exceptional case \eqref{exception},
which has already been treated in Step 1.

Outside this exceptional case, Corollary \ref{direct} gives
\begin{equation} \label{p0upper}
H_a(r)<A_{p_0}(w;1,1-r),\qquad 0<r<1,
\end{equation}
and the quotient $H_a/A_{p_0}$ is strictly decreasing.
Its limit as $r\to1^-$ is therefore strictly less than $1$.
By Lemma \ref{endpoint} and the strict increase of
$p\mapsto w^{1/p}$ on $(0,\infty)$, it follows that
$0<p_E<p_0\leq\frac12$.

Consequently, $-1<k_E=2p_E-1<0$, and Lemma \ref{hyp}
together with \eqref{S} shows that $S_{p_E}$ is strictly increasing.
At $p=p_E$, Corollary \ref{direct} yields $f/g>1$ if $a>0$
and $f/g<1$ if $a<0$.
In both cases, \eqref{ratio} gives
\begin{equation} \label{pElower}
A_{p_E}(w;1,1-r)<H_a(r),\qquad 0<r<1.
\end{equation}
Combining \eqref{p0upper} and \eqref{pElower}
proves \eqref{lowsharp}.
\end{proof}

\begin{remark}\label{rem-a1}
For $a=1$, Euler's integral gives
\begin{equation*}
H_1(r)=1-\frac{b}{c}r=A_1(w;1,1-r).
\end{equation*}
Thus $p_0=p_E=1$. By lemma \ref{meanmono}, for every $0<r<1$,
\begin{equation*}
A_\lambda(w;1,1-r)\leq A_1(w;1,1-r) 
\end{equation*}
is equivalent to $\lambda\le1,$
and
\begin{equation*}
A_1(w;1,1-r)\leq A_\mu(w;1,1-r) 
\end{equation*}
is equivalent to $\mu\geq 1.$
\end{remark}

\begin{proposition} \label{prop-log}
Assume $a=0$ and $c>b>0$.
\begin{enumerate}
\item If $c\geq \max(1-2a,2b)$ holds, then $0<p_0\leq p_E<1$. If $c=2b=1$, then $p_0=p_E=\frac{1}{2}$ and \eqref{identity} holds. Otherwise $p_0<p_E$ and \eqref{upsharp} holds.
\item If $c\le\min(1-2a,2b)$ holds, then $0<p_E\leq p_0<1$. The equality case is again exactly $c=2b=1$. Outside this case, \eqref{lowsharp} holds.
\end{enumerate}
\end{proposition}
\begin{proof}
At $p=p_0=\frac{c}{1+c}$, put $k_0=2p_0-1$. Since $S_{p_0}(0^+)=0$,
\begin{equation*}
S_{p_0}(u)=\frac{c}{2}[F_{k_0}(u)-k_0]-\frac{d}{2}Q_{k_0}(u).
\end{equation*}
In the upper branch, $0\leq k_0<1$ and $d\geq0$. Lemma \ref{hyp} gives $S_{p_0}(u)<0$ for $u>0$, unless $k_0=d=0$. Hence \eqref{JD} and \eqref{Dquot} give
\begin{equation*}
D_{p_0}'>0,\quad A_{p_0}(w;1,1-r)<H_0(r).
\end{equation*}
In the lower branch, $-1<k_0\leq0$ and $d\leq0$, so Lemma \ref{hyp} gives $S_{p_0}(u)>0$ for $u>0$, unless $k_0=d=0$. Therefore
\begin{equation*}
D_{p_0}'<0,\quad H_0(r)<A_{p_0}(w;1,1-r).
\end{equation*}

By \eqref{pE0} and \eqref{geomless}, $0<p_E<1$. In the nonexceptional upper branch, for every $u_*>0$,
\begin{equation*}
0<D_{p_0}(u_*)\leq\lim_{u\to\infty}D_{p_0}(u)=\log\frac{w^{\frac{1}{p_E}}}{w^{\frac{1}{p_0}}}.
\end{equation*}
Hence $p_0<p_E<1$ and $0<k_E=2p_E-1<1$. In the nonexceptional lower branch,
\begin{equation*}
\log\frac{w^{\frac{1}{p_E}}}{w^{\frac{1}{p_0}}}=\lim_{u\to\infty}D_{p_0}(u)\leq D_{p_0}(u_*)<0,
\end{equation*}
so $0<p_E<p_0\leq\frac{1}{2}$ and $-1<k_E=2p_E-1<0$. Lemma \ref{hyp} and \eqref{S} therefore give
\begin{equation*}
\begin{cases}
S_{p_E}'<0,&\text{in the upper branch},\\
S_{p_E}'>0,&\text{in the lower branch}.
\end{cases}
\end{equation*}
Lemma \ref{logturning} yields
\begin{equation*}
\begin{cases}
H_0(r)<A_{p_E}(w;1,1-r),&\text{in the upper branch};\\
A_{p_E}(w;1,1-r)<H_0(r),&\text{in the lower branch}
\end{cases}
\end{equation*}
for every $r\in(0,1)$
These inequalities, together with the $p_0$-bounds above, prove \eqref{upsharp} and \eqref{lowsharp}.

If $k_0=d=0$, then $c=2b=1$, $p_0=\frac{1}{2}$ and $S_{p_0}\equiv0$. Equation \eqref{JD} and the initial data give $D_{p_0}\equiv0$, and hence
\begin{equation*}
H_0=A_{\frac{1}{2}},\quad p_E=p_0=\frac{1}{2}.
\end{equation*}
This is the exceptional identity.
\end{proof}

\begin{proof}[\textbf{Proof of Theorem \ref{thm1}}]
By Remark \ref{rem-a1}, we only consider $a<1$ below. The necessity of the restrictions has already been proved in Corollaries \ref{localsharp} and \ref{endpointsharp}: in the upper branch they give $\lambda\leq p_0$ and $\mu\geq p_E$, while in the lower branch they give $\lambda\leq p_E$ and $\mu\geq p_0$. Propositions \ref{prop-nonzero} and \ref{prop-log} give the order relations, the sharp endpoint comparisons, and the exceptional identity. Thus it remains only to prove sufficiency.

Assume first $c\geq \max(1-2a,2b)$. If $\lambda\leq p_0$ and $\mu\geq p_E$, Lemma \ref{meanmono} and the sharp bounds give
\begin{equation*}
\begin{aligned}
A_\lambda(w;1,1-r)&\leq A_{p_0}(w;1,1-r)\leq H_a(r)\\
&\leq A_{p_E}(w;1,1-r)\leq A_\mu(w;1,1-r)
\end{aligned}
\end{equation*}
for every $r\in(0,1)$. This proves sufficiency in the upper branch, including the exceptional identity.

Assume next $c\le\min(1-2a,2b)$. If $\lambda\leq p_E$ and $\mu\geq p_0$, then
\begin{equation*}
\begin{aligned}
A_\lambda(w;1,1-r)&\leq A_{p_E}(w;1,1-r)\leq H_a(r)\\
&\leq A_{p_0}(w;1,1-r)\leq A_\mu(w;1,1-r)
\end{aligned}
\end{equation*}
for every $r\in(0,1)$. The strict inequalities and the equality case are exactly those established in Propositions \ref{prop-nonzero} and \ref{prop-log}. This proves sufficiency in the lower branch and completes the proof.
\end{proof}

\section{The symmetric range \texorpdfstring{$a>1$}{a>1}} \label{sec-4}
Corollaries \ref{localsharp} and \ref{endpointsharp} already show that $\lambda\leq p_E$ and $\mu\geq p_0$ are necessary. The symmetric condition $c=2b$ eliminates the $d$-dependent terms from the differential operator and the residual, and places both sharp orders above $1$. We therefore establish the sharp endpoint comparisons and their order relation, after which Lemma \ref{meanmono} yields sufficiency in Theorem \ref{thm2}.

\begin{proposition} \label{prop-sym}
Let $a>1$, $b>0$ and $c=2b$. Then
\begin{equation*}
1<p_E<p_0<a,
\end{equation*}
and
\begin{equation*}
A_{p_E}\left(\frac{1}{2};1,1-r\right)<\F(-a,b;2b;r)^{\frac{1}{a}}<A_{p_0}\left(\frac{1}{2};1,1-r\right),\quad 0<r<1.
\end{equation*}
\end{proposition}
\begin{proof}
Since $c=2b$, recalling \eqref{tanh}:
\begin{equation*}
w=\frac{1}{2},\quad d=0,\quad v_0=0.
\end{equation*}
Since $a>1$, strict Jensen inequality gives
\begin{equation*}
C_a=\int_0^1(1-t)^a\mathrm{d}\nu(t) > \left(\int_0^1(1-t)\mathrm{d}\nu(t)\right)^a=2^{-a}.
\end{equation*}
On the other hand,
\begin{equation*}
C_a<\int_0^1(1-t)\mathrm{d}\nu(t)=\frac{1}{2}.
\end{equation*}
Thus $2^{-a}<C_a<2^{-1}$. Since $C_a=2^{-\frac{a}{p_E}}$ and $x\mapsto2^{-x}$ is strictly decreasing,
\begin{equation} \label{roughsym}
1<p_E<a.
\end{equation}
Furthermore, by the definition of $p_0$ in \eqref{orders-1}:
\begin{equation} \label{p0sym}
1<p_0<a.
\end{equation}

At $p=p_0$, put $k_0=2p_0-1$. Then
\begin{equation*}
k_0>1,\quad S_{p_0}(0^+)=0,\quad S_{p_0}(u)=p_0-a+b\left(\frac{\sinh(k_0u)}{\sinh u}-1\right).
\end{equation*}
Lemma \ref{hyp} gives $S_{p_0}(u)>0$ for $u>0$. Hence Corollary \ref{direct} gives
\begin{equation} \label{symupper}
\F(-a,b;2b;r)^{\frac{1}{a}}<A_{p_0}\left(\frac{1}{2};1,1-r\right),\quad 0<r<1,
\end{equation}
and the quotient of the left-hand side by the right-hand side is strictly decreasing. Therefore, for any $r_*\in(0,1)$,
\begin{equation*}
\begin{aligned}
\frac{2^{-\frac{1}{p_E}}}{2^{-\frac{1}{p_0}}}
&=\lim_{r\to1^-}\frac{\F(-a,b;2b;r)^{\frac{1}{a}}}{A_{p_0}(\frac{1}{2};1,1-r)}\\
&\leq\frac{\F(-a,b;2b;r_*)^{\frac{1}{a}}}{A_{p_0}(\frac{1}{2};1,1-r_*)}<1.
\end{aligned}
\end{equation*}
Hence $p_E<p_0$. Together with \eqref{roughsym} and \eqref{p0sym}, this proves $1<p_E<p_0<a$.

At $p=p_E$, one has $k_E=2p_E-1>1$ and $S_{p_E}'>0$. Since $a>0$, Corollary \ref{direct} and \eqref{ratio} give
\begin{equation} \label{symlower}
A_{p_E}\left(\frac{1}{2};1,1-r\right)<\F(-a,b;2b;r)^{\frac{1}{a}},\quad 0<r<1.
\end{equation}
Combining \eqref{symlower} and \eqref{symupper} proves the stated sharp two-sided estimate.
\end{proof}

\begin{proof}[\textbf{Proof of Theorem \ref{thm2}}]
The necessity of $\lambda\leq p_E$ and $\mu\geq p_0$ has already been proved in Corollaries \ref{endpointsharp} and \ref{localsharp}, respectively. Proposition \ref{prop-sym} gives \eqref{symrange} and \eqref{symsharp}. For sufficiency, assume $\lambda\leq p_E$ and $\mu\geq p_0$. Lemma \ref{meanmono} then yields
\begin{equation*}
\begin{aligned}
A_\lambda\left(\frac{1}{2};1,1-r\right)&\leq A_{p_E}\left(\frac{1}{2};1,1-r\right)<\F(-a,b;2b;r)^{\frac{1}{a}}\\
&<A_{p_0}\left(\frac{1}{2};1,1-r\right)\leq A_\mu\left(\frac{1}{2};1,1-r\right).
\end{aligned}
\end{equation*}
Thus the two restrictions are sufficient, and the proof is complete.
\end{proof}

\appendix
\section{Auxiliary proofs}
This appendix collects the standard power-mean monotonicity result used throughout the paper and the proofs of the two auxiliary criteria whose statements are retained in Section \ref{sec-2}.

\begin{lemma}\label{meanmono}
Fix $w\in(0,1)$ and $h>0$ with $h\neq1$. Then $p\mapsto A_p(w;1,h)$ is strictly increasing on $\mathbb{R}$.
\end{lemma}
\begin{proof}
Let $0<p_1<p_2$. Since $x\mapsto x^{\frac{p_2}{p_1}}$ is strictly convex,
\begin{equation*}
w+(1-w)h^{p_2}>[w+(1-w)h^{p_1}]^{\frac{p_2}{p_1}},
\end{equation*}
and hence $A_{p_1}(w;1,h)<A_{p_2}(w;1,h)$.

For $p_2>0$, strict concavity of $\log$ gives
\begin{equation*}
\log[w+(1-w)h^{p_2}]>(1-w)p_2\log h,
\end{equation*}
so $A_0(w;1,h)<A_{p_2}(w;1,h)$. If $p_1<0$, division of
\begin{equation*}
\log[w+(1-w)h^{p_1}]>(1-w)p_1\log h
\end{equation*}
by $p_1$ yields $A_{p_1}(w;1,h)<A_0(w;1,h)$.

If $p_1<p_2<0$, then $-p_1>-p_2>0$ and
\begin{equation*}
A_{p_1}(w;1,h)=A_{-p_1}(w;1,h^{-1})^{-1},\quad
A_{p_2}(w;1,h)=A_{-p_2}(w;1,h^{-1})^{-1}.
\end{equation*}
The positive-order case gives $A_{p_1}(w;1,h)<A_{p_2}(w;1,h)$.
\end{proof}

\begin{proof}[Proof of Lemma \ref{hyp}]
The function $x\mapsto x\coth x$ is strictly increasing on $(0,\infty)$, since
\begin{equation*}
\frac{\mathrm{d}}{\mathrm{d}x}(x\coth x)
=\frac{\sinh x\cosh x-x}{\sinh^2x}>0,
\end{equation*}
and the numerator vanishes at $x=0$ and has derivative $2\sinh^2x>0$. For $k>0$,
\begin{equation*}
\frac{F_k'(u)}{F_k(u)}
=k\coth(ku)-\coth u
=\frac{(ku)\coth(ku)-u\coth u}{u},
\end{equation*}
which gives the asserted monotonicity for $0<k<1$ and $k>1$. If $-1<k<0$, then $F_k=-F_{-k}$; the cases $k=0,1$ are immediate.

For $|k|<1$,
\begin{equation*}
Q_k'(u)=\frac{\cosh(ku)\cosh u-k\sinh(ku)\sinh u-1}{\sinh^2u}.
\end{equation*}
The numerator vanishes at $u=0$ and has derivative
\begin{equation*}
(1-k^2)\cosh(ku)\sinh u>0.
\end{equation*}
Hence $Q_k'(u)>0$ for $u>0$.
\end{proof}

\begin{proof}[Proof of Proposition \ref{comp}]
Since $m'=Pm$, a direct computation gives
\begin{equation*}
W'=-m\tilde f\tilde g\,\sigma,
\qquad
R'=\frac{W}{m\tilde g^2}.
\end{equation*}
If $\sigma<0$, then $W'>0$. Since $W(0^+)=0$, one has $W>0$, hence $R'>0$ and $R>1$. 
The case $\sigma>0$ is identical with all inequalities reversed.

Suppose now that $R(\infty)=1$ and that $\sigma$ changes sign from negative to positive at $u_0$. Then $W'>0$ on $(0,u_0)$ and $W'<0$ on $(u_0,\infty)$. Thus $W>0$ on $(0,u_0]$, and $W$ is strictly decreasing after $u_0$. It must cross a unique zero. Otherwise $W\geq0$ on $(0,\infty)$, so $R\geq R(0^+)=1$ and $R$ is strictly increasing on $(0,u_0]$, contradicting $R(\infty)=1$. Therefore $R$ first increases and then decreases, and the equal endpoint values imply $R(u)>1$ for every $u>0$. If the signs of $\sigma$ are reversed, the same argument with all inequalities reversed gives $R(u)<1$.
\end{proof}

\end{document}